\documentclass{article}
\usepackage[]{amsthm}
\usepackage[]{amsmath}
\usepackage[]{times}
\usepackage[]{hyperref}

\newtheorem*{proposition}{Proposition}
\newtheorem*{corollary}{Corollary}

\title{Integral Recurrences from A to Z}
\author{Robert Dougherty-Bliss \\ {\small \href{mailto:robert.w.bliss@gmail.com}{robert.w.bliss@gmail.com}}}
\date{\today}

\begin{document}
\maketitle

\begin{abstract}
    \noindent George Boros and Victor Moll's masterpiece \emph{Irresistible
    Integrals} does well to include a suitably-titled appendix, ``The
    Revolutionary WZ Method,'' which gives a brief overview of the celebrated
    Wilf--Zeilberger method of definite summation. Paradoxically,
    \emph{Irresistible Integrals} does \emph{not} contain the suitably-titled
    appendix, ``The Revolutionary AZ Method,'' which would have been an
    excellent place to give a brief overview of the Almkvist--Zeilberger method
    of definite \emph{integration}! This omission can be forgiven, but once
    realized it must be rectified. The remarkable AZ machinery deserves to be
    more widely known to the general public than it is. We will do our part by
    presenting a series of case studies that culminate in a---fun but
    overkill---integral-based proof that $e$ is irrational.
\end{abstract}

\begin{center}
    {\it
    Behold, I will stand before thee there upon the rock in Horeb; and thou
    shalt smite the rock and there shall come water out of it, that the people
    may drink.}

    \hfill --- Exodus 17:6
\end{center}

\noindent You have been up all night working out the masterpiece solution to
your latest problem. Your answer depends on the integral sequence
\begin{equation*}
    I(n) = \int_{-\infty}^\infty \frac{x^{2n}}{(x^2 + 1)^{n + 1}}\ dx,
\end{equation*}
which you desperately need to evaluate. You know that you could break out
special functions, contour integrals, or some other method, but you would
really just like a quick answer without much fuss.

You run to download the file \texttt{EKHAD} from
\begin{center}
    \url{https://sites.math.rutgers.edu/~zeilberg/tokhniot/EKHAD}
\end{center}
and read it into Maple with ``\texttt{read EKHAD;}''. You type the command
\begin{verbatim}
    AZd(x^(2 * n) / (x^2 + 1)^(n + 1), x, n, N);
\end{verbatim}
and hardly a second has passed when Maple produces the following:
\begin{verbatim}
                  -2 n - 1 + (2 n + 2) N, -x
\end{verbatim}
You cry out in joy, for the \emph{Almkvist--Zeilberger algorithm} has told you
that your integrand satisfies the ``recurrence''
\begin{equation*}
    (-2n - 1 + (2n + 2) N) \frac{x^{2n}}{(x^2 + 1)^{n + 1}} = -\frac{d}{dx}x \frac{x^{2n}}{(x^2 + 1)^{n + 1}},
\end{equation*}
where $N$ is the shift operator defined by $N f_n(x) = f_{n + 1}(x)$.
Integrating this equation on $(-\infty, \infty)$ gives the identity
\begin{equation*}
    (-2n - 1 + (2n + 2) N) I(n) = 0,
\end{equation*}
which would traditionally be written as
\begin{equation*}
    I(n + 1) = \frac{2n + 1}{2(n + 1)} I(n).
\end{equation*}
You are well-aware that the sequence
\begin{equation*}
    \frac{\pi}{4^n} {2n \choose n}
\end{equation*}
satisfies the same recurrence and initial condition, so you have just proven
that
\begin{equation*}
    I(n) = \int_{-\infty}^\infty \frac{x^{2n}}{(x^2 + 1)^{n + 1}}\ dx = \frac{\pi}{4^n} {2n \choose n}
\end{equation*}
with minimal effort on your part. Such is a normal case study of the \emph{AZ}
algorithm.

In general, we often want to understand the sequence of \emph{definite}
integrals
\begin{equation*}
    I(n) = \int F_n(x)\ dx.
\end{equation*}
Perhaps we would like to compute the first twenty terms or so to see what
$I(n)$ looks like. Sometimes we can ask a computer to churn these out, but
other times $F_n(x)$ is so complicated that even our electronic friends would
struggle to keep up for large $n$. What we need is an \emph{efficient
algorithm} to compute the terms of $I(n)$. We need a \emph{recurrence}.

There are plenty of ad-hoc methods to find a recurrence for $I(n)$. You could
integrate by parts or differentiate under the integral sign, for example. But
these all require ingenuity, insight, and hard work. As Sir Alfred Whitehead
once remarked, such ingenuity is overrated. No one wants to work hard---we want
answers!

The \emph{painless} way to discover these recurrences for large classes of
integrals is the \emph{Almkvist--Zeilberger} algorithm. This is the direct
analog of the celebrated Wilf--Zeilberger method of automatic definite
summation, but it has received less attention than its discrete counterpart.
Our goal here is to explore the Almkvist--Zeilberger algorithm with a few case
studies, leaving the door open for more experimentation.

\section{A quickstart guide to the AZ algorithm}%
\label{sec:a_quickstart_guide_to_the_az_algorithm}

The Wilf--Zeilberger method of definite summation is a breakthrough in
automatic summation techniques. Roughly, the Wilf--Zeilberger method can
automatically prove (and semi-automatically discover) most commonly occurring
summation identites of the form
\begin{equation*}
    S(n) = \sum_k f(n, k) = RHS(n).
\end{equation*}
One piece of the puzzle is that, whenever $f(n, k)$ is a ``suitable'' function,
it satifies a specific type of inhomogenous linear recurrence with polynomial
coefficients in $n$. Exactly, there exists a nonnegative integer $d$ and
polynomials $p_j(n)$ such that
\begin{equation*}
    \sum_{j = 0}^d p_j(n) f(n + j, k) = G(n, k + 1) - G(n, k),
\end{equation*}
where $G(n, k)$ is some function with $G(n, \pm \infty) = 0$. Summing over $k$
yields the recurrence
\begin{equation*}
    \sum_{j = 0}^d p_j(n) S(n + j) = G(n, \infty) - G(n, -\infty) = 0.
\end{equation*}
This method has been (rightly) advertised from here to the Moon and back. See
the article, \cite{whatis}, the book \cite{ab}, the lecture notes
\cite{recitations}, and the lively \emph{Monthly} article \cite{computer}.

The Almkvist--Zeilberger algorithm is to definite integrals what the
Wilf--Zeilberger method is to definite sums. The input to the algorithm is a
``suitable'' function $F_n(x)$ with a discrete parameter $n$. The output is a
linear recurrence operator $L(N, n)$ with polynomial coefficients in $N$ and
$n$, and function $R(n, x)$, rational in $n$ and $x$, such that
\begin{equation*}
    L(N, n) F_n(x) = \frac{d}{dx} R(n, x) F_n(x).
\end{equation*}
Explicitly, there is a nonnegative integer $d$ and polynomials $p_k(n)$ such
that
\begin{equation*}
    \sum_{k = 0}^d p_k(n) F_{n + k}(x) = \frac{d}{dx} R(n, x) F_n(x).
\end{equation*}
The left-hand side is independent of $x$ except for the $F_n(x)$, so
integrating this equation on $[0, 1]$, say, gives
\begin{equation*}
    L(N, n) \int_0^1 F_n(x) = R(n, 1) F_n(1) - R(n, 0) F_n(0).
\end{equation*}
If $F_n(0) = F_n(1) = 0$ and $R(n, x)$ is well-behaved, then $I(n) = \int_0^1
F_n(x)$ satisfies
\begin{equation*}
    L(N, n) I(n) = 0,
\end{equation*}
meaning that we have discovered a recurrence for the sequence of integrals
$I(n)$. The only thing to verify is that $F_n(x)$ is ``suitable,'' and that
$R(n, x)$ is well-behaved on the region of integration.

What functions are ``suitable''? The requirement is that $F_n(x)$ is
\emph{hypergeometric} in $n$ and $x$, meaning that there exist fixed rational
functions $R_1(n, x)$ and $R_2(n, x)$ such that
\begin{align*}
    F_{n + 1}(x) / F_n(x) &= R_1(n, x) \\
    F_n'(x) / F_n(x) &= R_2(n, x).
\end{align*}
This is all that the algorithm needs to produce its identity.

The version of the Almkvist--Zeilberger algorithm that we will use is
implemented in the procedure \texttt{AZd(f, x, n, N)} in the Maple package
\texttt{EKHAD} referenced in the introduction. It takes an expression $f$ in
the continuous variable $x$ and discrete parameter $n$. The symbol $N$ stands
for the ``shift'' operator $N$ on the set of sequences by
\begin{equation*}
    N a(n) = a(n + 1).
\end{equation*}
For example, the Fibonacci numbers $F(n)$ satisfy
\begin{equation*}
    (N^2 - N - 1) F(n) = 0.
\end{equation*}

Now, let us get on to the case studies.

\section{Factorials}%
\label{sub:factorials}

Let us begin humbly, by evaluating an integral that we already know.
\begin{proposition}
    For each integer $n \geq 0$,
        \begin{equation*}
            I(n) = \int_0^\infty e^{-x} x^n\ dx = n!.
        \end{equation*}
\end{proposition}

\begin{proof}
    Type the following into Maple:
    \begin{verbatim}
            AZd(exp(-x) * x^n, x, n, N);
    \end{verbatim}
    This produces:
    \begin{verbatim}
            N - n - 1, -x
    \end{verbatim}
    That is, the Almkvist--Zeilberger algorithm has told us that
    \begin{equation}
        \label{factorial}
        (N - (n + 1)) f_n(x) = -\frac{d}{dx} e^{-x} x^{n + 1}.
    \end{equation}
    Since the antiderivative of the right-hand side vanishes for $x = 0$ and $x
    = \infty$, integrating on $[0, \infty)$ gives
    \begin{equation*}
        (N - (n + 1)) I(n) = 0,
    \end{equation*}
    and since $I(0) = 1$, we have $I(n) = n!$.
\end{proof}

\section{``A Complicated Integral''}

This is from Section~3.8 of \cite{integrals}.

\begin{proposition}
    \begin{equation*}
        I(n) = \int_0^\infty \frac{x^n}{(x + 1)^{n + r + 1}}\ dx = \left[ r {r + n \choose n} \right]^{-1}.
    \end{equation*}
\end{proposition}

\begin{proof}
    Type the following into Maple:
    \begin{verbatim}
        AZd(x^n / (x + 1)^(n + r + 1), x, n, N);
    \end{verbatim}
    This produces:
    \begin{verbatim}
        (n + 1) + (-n - r - 1) N, x
    \end{verbatim}
    And for $r > 0$, integrating the implied identity
    \begin{equation}
        ((n + 1) - (n + r + 1) N) \frac{x^n}{(x + 1)^{n + r + 1}} = \frac{d}{dx} x \frac{x^n}{(x + 1)^{n + r + 1}}
    \end{equation}
    yields
    \begin{equation*}
        ((n + 1) - (n + r + 1) N) I(n) = 0.
    \end{equation*}
    The sequence $(r {r + n \choose n})^{-1}$ satisfies the same recurrence and
    initial condition (check!).
\end{proof}

\section{Central binomial coefficients}%
\label{sub:central_binomial_coefficients}

\begin{proposition}
    The integral sequence
    \begin{equation*}
        I(n) = \int_0^1 (x (1 - x))^n\ dx
    \end{equation*}
    satisfies
    \begin{equation*}
        (N - \frac{n + 1}{2(2n + 3)}) I(n) = 0.
    \end{equation*}
\end{proposition}

\begin{proof}
    Type the following into Maple:
    \begin{verbatim}
            AZd((x * (1 - x))^n, x, n, N);
    \end{verbatim}
    This produces:
    \begin{verbatim}
            n + 1 + (-4 n - 6) N, (-1 + 2 x) (-1 + x) x
    \end{verbatim}
    Integrating the implied identity
    \begin{equation}
        \label{central}
        (N - \frac{n + 1}{2(2n + 3)}) (x (1 - x))^n = \frac{d}{dx} (2x - 1) (x - 1) x (x (1 - x))^n,
    \end{equation}
    on $[0, 1]$ yields the result, since the antiderivative of the right-hand
    side vanishes at $x = 0$ and $x = 1$.
\end{proof}

The recurrence implies that $I(n)$ begins as follows:
\begin{equation*}
    1/6,\ 1/30,\ 1/140,\ 1/630,\ 1/2772,\ 1/12012,\ 1/51480, \dots
\end{equation*}

\begin{corollary}
    \begin{equation*}
        I(n) = \frac{1}{(2n + 1) {2n \choose n}}
    \end{equation*}
\end{corollary}

\begin{proof}
    Both sequences satisfy the same recurrence and initial condition (check!).
\end{proof}

This fact gives us an interesting identity. One way to \emph{try} and evaluate
$I(n)$ is by applying the binomial theorem to the integrand:
\begin{align*}
    I(n) &= \int_0^1 x^n (1 - x)^n\ dx \\
         &= \int_0^1 \sum_{k = 0}^n {n \choose k} (-1)^k x^{n + k}\ dx \\
         &= \sum_{k = 0}^n {n \choose k} \frac{(-1)^k}{n + k + 1}.
\end{align*}
This remaining sum is complicated, but we can pair it with our previous corollary to get another.

\begin{corollary}
    \begin{equation*}
        \sum_{k = 0}^n {n \choose k} \frac{(-1)^k}{n + k + 1} = \frac{1}{(2n + 1) {2n \choose n}}
    \end{equation*}
\end{corollary}

\section{Irrationality and Euler's Constant}
\label{sub:euler_s_constant}

Our final case study is a slightly more complicated sequence of integrals. We
will not be able to derive a closed form, but we will derive a wealth of other
information.

\begin{proposition}
    The integral sequence
    \begin{equation*}
        I(n) = \int_0^1 (x (1 - x))^n e^{-x}\ dx
    \end{equation*}
    satisfies
    \begin{equation*}
        (N^2 + 2 (2n + 3) (n + 2) N - (n + 1)(n + 2)) I(n) = 0.
    \end{equation*}
\end{proposition}

\begin{proof}
    Let $f_n(x)$ be the integrand. The Almkvist--Zeilberger algorithm produces
    the ``calculus exercise''
    \begin{align}
        \label{euler}
        \begin{split}
            (N^2 + 2(2n + 3) (n + 2) N &- (n + 1)(n + 2) f_n(x)
                \\ &= \frac{d}{dx} (-2n x^3 - x^4 + 3 n x^2 - 2x^3 - nx + 5 x^2 - 2x) f_n(x),
        \end{split}
    \end{align}
    and integrating this proves the proposition.
\end{proof}

The recurrence is hopelessly complicated; we probably won't be able to solve
it. But it does produce the following initial terms:
\begin{equation*}
    -1+\frac{3}{e},\ 14 - \frac{38}{e},\ -426 + \frac{1158}{e},\ 24024 - \frac{65304}{e}, \dots
\end{equation*}
This data is very suggestive! It leads us to conjecture that
\begin{equation*}
    I(n) = a_n + b_n e^{-1}
\end{equation*}
for some \emph{integers} $a_n$ and $b_n$. This is true by virtue of the
recurrence: if $I(n) = a_n + b_n e^{-1}$ and $I(n + 1) = a_{n + 1} + b_{n + 1}
e^{-1}$, then
\begin{align*}
    I(n + 2) &= -2(2n + 3)(n + 2) I(n + 1) + (n + 1)(n + 2) I(n) \\
             &= -2(2n + 3)(n + 2) (a_{n + 1} + b_{n + 1} e^{-1}) + (n + 1)(n + 2) (a_n + b_n e^{-1}) \\
             &= a_{n + 2} + b_{n + 2} e^{-1},
\end{align*}
where we take
\begin{align*}
    a_{n + 2} &= -2(2n + 3)(n + 2) a_{n + 1} + (n + 1)(n + 2) a_n \\
    b_{n + 2} &= -2(2n + 3)(n + 2) b_{n + 1} + (n + 1)(n + 2) b_n.
\end{align*}
That is, $a_n$ and $b_n$ are sequences of integers which satisfy the \emph{same
recurrence} that $I(n)$ satisfies, only the initial conditions are different:
\begin{align*}
    a_1 = -1 &\quad a_2 = 14 \\
    b_1 = 3 &\quad b_2 = -38.
\end{align*}
Better yet, note that
\begin{align*}
    -\frac{a_4}{b_4} &= \frac{24024}{65304} \\
                     &= 0.36787945\dots \\
                     &\approx e^{-1}.
\end{align*}
That is, $-a_n / b_n$ seems to be a good approximation to $e^{-1}$!

To see why this is, we must go back to the initial integral. For $0 \leq x \leq
1$, we have $x(1 - x) \leq 1/4$, therefore
\begin{equation*}
    0 \leq I(n) = \int_0^1 e^{-x} (x (1 - x))^n\ dx \leq \frac{1}{4^n} \int_0^1 e^{-x}\ dx,
\end{equation*}
which shows that $I(n)$ goes to zero exponentially quickly. Therefore
\begin{equation*}
    |a_n + b_n e^{-1}| \to 0
\end{equation*}
exponentially quickly, meaning that
\begin{equation*}
    |\frac{a_n}{b_n} + e^{-1}| = |(-\frac{a_n}{b_n}) - e^{-1}| \to 0.
\end{equation*}
In words, $-a_n / b_n$ gives an exponentially-good rational approximation of
$e^{-1}$. To double check, we can use our recurrence to compute $a_{20} /
b_{20}$:
\begin{align*}
    -\frac{a_{20}}{b_{20}} &= \frac{493294164866383351699429534601141833239920640000}{1340912564441170249019237618446466016434749440000} \\
                          &=       0.3678794411714423215955237701614608674\dots \\
                          &\approx 0.367879441171442321595523770161460867445\dots \\
                          &= e^{-1}.
\end{align*}
\emph{Better still}, this remarkable approximation $-a_n / b_n \approx e^{-1}$
is \emph{too good to be true} in the following sense.

\begin{proposition}
    Let $\alpha$ be a real number. If there exist sequences of integers $a_n$
    and $b_n$ such that $|b_n| \to \infty$ and
    \begin{equation*}
        |\alpha - \frac{a_n}{b_n}| \leq \frac{C}{|b_n|^{1 + \delta}}
    \end{equation*}
    for some positive constants $C$ and $\delta$, then $\alpha$ is irrational.
\end{proposition}

\begin{proof}
    If $\alpha = a / b$ is rational, then
    \begin{equation*}
        |\alpha - \frac{a_n}{b_n}| = \frac{|(b_n a - b a_n) / b|}{|b_n|} \geq \frac{C'}{|b_n|}
    \end{equation*}
    for some positive constant $C'$. But the inequality
    \begin{equation*}
        \frac{C'}{|b_n|} \leq \frac{C}{|b_n|^{1 + \delta}}
    \end{equation*}
    is impossible if $|b_n| \to \infty$.
\end{proof}

This fact together with our approximation $-a_n / b_n \approx e^{-1}$ gives us
an unnecessarily complicated proof that $e$ is irrational.

\begin{proposition}
    $e$ is irrational with $\delta = 1$.
\end{proposition}

\begin{proof}
    Let $a_n$ and $b_n$ be the approximating sequences induced by
    \begin{equation*}
        I(n) = \int_0^1 e^{-x} (x(1 - x))^n\ dx.
    \end{equation*}
    We have
    \begin{equation*}
        |a_n + b_n e^{-1}| \leq \frac{1}{4^n} \int_0^1 e^{-x} = \frac{C}{4^n}.
    \end{equation*}
    The sequence $b_n$ satisfies the recurrence
    \begin{equation*}
        (N^2 + 2 (2n + 3) (n + 2) N - (n + 1)(n + 2)) I(n) = 0.
    \end{equation*}
    It turns out---see \cite{asymptotics}---that this reveals considerable
    asymptotic information about $b_n$. In particular, if we rewrite the
    recurrence as a polynomial in $n$, the leading coefficient is $4N - 1$. The
    only solution to $4N - 1 = 0$ is $N = 1/4$, and this implies that $1/4^n
    \leq C' \frac{1}{|b^n|}$ for some constant $C'$. Thus
    \begin{equation*}
        |a_n + b_n e^{-1}| \leq \frac{C'}{|b_n|},
    \end{equation*}
    or
    \begin{equation*}
        |\frac{a_n}{b_n} + e^{-1}| \leq \frac{C'}{|b_n|^{1 + \delta}},
    \end{equation*}
    where $\delta = 1$. The claim follows from the previous proposition.
\end{proof}

Proving that $e$ is irrational is an easy exercise, but our proof gives a
\emph{quantitative measure} on the irrationality of $e$. Given a real $\alpha$,
the \emph{irrationality measure} of $\alpha$, denoted $\mu(\alpha)$ is defined
to be the smallest real $\mu$ such that
\begin{equation*}
    \left| \alpha - \frac{p}{q} \right| > \frac{1}{q^{\mu + \epsilon}}
\end{equation*}
holds for any $\epsilon > 0$ and all integers $p$ and $q$ with $q$ sufficiently
large. A number is irrational iff it has irrationality measure $> 1$. In fact,
irrationality measures are $1$ for rational numbers, $2$ for algebraic numbers,
and $\geq 2$ for transcendental numbers. It is well-known that almost every
number has irrationality measure $2$, but it is notoriously difficult to
\emph{prove} this for naturally ocurring, \emph{specific} constants.

Our proof implies the obvious lower bound $\mu(e) \geq 1 + \delta = 2$. With a
little more work---see \cite{poorten}---we can show that $\mu(e) \leq 1 + 1 /
\delta = 2$, so $\mu(e) = 2$.

It has become a \emph{game} to provide better and better upper bounds for the
irrationality measure of famous constants. For example, $\mu(\pi) \geq 2$ since
$\pi$ is transcendental, but we do not know the exact value. The current
``world record'' upper bound is held by Zeilberger and Zudilin, who showed in
\cite{pi} that
\begin{equation*}
    \mu(\pi) \leq 7.103205334137\dots
\end{equation*}
Ignoring the many technical details, their proof is very similar to ours. The
basic idea is to find a rapidly-decaying sequence $I(n)$ such that
\begin{equation*}
    I(n) = a_n + \pi b_n
\end{equation*}
for integers $a_n$ and $b_n$, then show that $b_n$ has nice asymptotic
properties. This style of proof was notably used by Fritz Beukers in
\cite{beukers} where he elegantly proved that the constants
\begin{equation*}
    \zeta(2) = \sum_{k \geq 1} \frac{1}{k^2}
\end{equation*}
and
\begin{equation*}
    \zeta(3) = \sum_{k \geq 1} \frac{1}{k^3}
\end{equation*}
are irrational by considering integrals of the form
\begin{equation*}
    \int_0^1 \int_0^1 \frac{x^n y^n (1 - x)^n (1 - y)^n}{1 - xy}\ dx\ dy
\end{equation*}
and
\begin{equation*}
    \int_0^1 \int_0^1 \int_0^1
        \frac{x^n y^n z^n (1 - x)^n (1 - y)^n (1 - z)^n}{(1 - (1 - xy)z)^{n + 1}}\ dx\ dy\ dz,
\end{equation*}
respectively. (The irrationality of $\zeta(3)$ was first shown in stunning
fashion by Roger Ap\'ery; see \cite{poorten} and \cite{apery}.)

To find their approximating sequences, Zeilberger and Zudilin tweaked integrals
similar to the ones above, adding parameters to the integrands and performing
an exhaustive computer search to find those parameters which gave the
empirically best upper bound. This method continues to provide possible avenues
for constructive irrationality proofs; see \cite{me} and \cite{automatic}.

It is too late for us to become famous proving that $\zeta(3)$ is irrational.
In fact, nothing we have done here is ``new'' or ``groundbreaking.'' We should
be content to have some new tools to play with. But you never know: One day you
might just plug the right integrand into the Almkvist--Zeilberger algorithm to
prove that
\begin{center}
    (FAMOUS CONSTANT)
\end{center}
is irrational.

Until then, have fun!

\end{document}